\documentclass{amsart}
\usepackage{amsfonts, amsmath, amscd}
\usepackage[psamsfonts]{amssymb}
\usepackage{mathrsfs}
\usepackage{pb-diagram}
\usepackage[usenames]{color}
\usepackage[all]{xy}
\usepackage{graphicx}
\usepackage{array}
\usepackage{orcidlink} 
\newtheorem{theorem}{Theorem}[section]
\newtheorem{lemma}[theorem]{Lemma}
\newtheorem{fact}[theorem]{Fact}
\newtheorem{proposition}[theorem]{Proposition}

\newtheorem{corollary}[theorem]{Corollary}
\newtheorem{em-example}[deff]{Example}
\newtheorem{problem}{Problem}
\theoremstyle{definition}
\newtheorem{definition}[theorem]{Definition}
\newtheorem{remark}[theorem]{Remark}

\newenvironment{example}{\begin{em-example} \em }{ \end{em-example}}

\begin{document}

\vspace{0.5in}

%%%%%%%%%%%%%%%%
\newcommand{\abs}[1]{\lvert#1\rvert}
\def\norm#1{\left\Vert#1\right\Vert}
\def\Q {{\Bbb Q}}
\def\I {{\Bbb I}}
\def\C {{\Bbb C}}
\def\N{{\Bbb N}}
\def\R{{\Bbb R}}
\def\di{{\mathrm{di}}}
\def\Z {{\Bbb Z}}
\def\U{{\Bbb U}}
\def\F{{\mathrm{E}}}
\def\Un{{\mathcal{U}}}
\def\Is{{\mathrm{Is}}\,}
\def\Aut{{\mathrm {Aut}}\,}
\def\supp{{\mathrm {supp}}\,}
\def\Homeo{{\mathrm{Homeo}}\,}
\def\gr{{\underline{\Box}}}
\def\diam{{\mathrm{diam}}\,}
\def\d{{\mathrm{dist}}}
\def\H{{\mathcal H}}
\def\me{{\mathrm{me}}}
\def\mO{{\overline{o}}}

\def\a{\alpha}
\def\d{\delta}
\def\D{\Delta}
\def\g{\gamma}
\def\s{\sigma}
\def\Si{\Sigma}
\def\implies{\Rightarrow}
\def\o{\omega}
\def\O{\Omega}
\def\G{\Gamma}

\def\sB{{\mathcal B}}
\def\sC{{\mathcal C}}
\def\sE{{\mathcal E}}
\def\sF{{\mathcal F}}
\def\sG{{\mathcal G}}
\def\sH{{\mathcal H}}
\def\sJ{{\mathcal J}}
\def\sK{{\mathcal K}}
\def\sL{{\mathcal L}}
\def\sM{{\mathcal M}}
\def\sN{{\mathcal N}}
\def\sO{{\mathcal O}}
\def\sP{{\mathcal P}}
\def\sR{{\mathcal R}}
\def\sS{{\mathcal S}}
\def\sT{{\mathcal T}}
\def\sU{{\mathcal U}}
\def\sV{{\mathcal V}}

\def\sbs{\subset}
\def\rar{\rightarrow}
\def\e{\epsilon}

\def\ti{\times}
\def\obr{^{-1}}
\def\stm{\setminus}
\def\newline{\hfill\break}

\def\Exp{{\mathrm{Exp}}\,}
\def\Iso{{\mathrm{Iso}}\,}
\def\Sym{{\mathrm{Sym}}\,}

%=================================================================

\title[Asplund spaces and the finest locally convex topology]
{Asplund locally convex spaces and the finest locally convex topology}
\author[J. K\c{a}kol and A. Leiderman]
%{J. K\c{a}kol and A. Leiderman}
{
Jerzy K\c{a}kol
\orcidlink{0000-0002-8311-2117}
and Arkady Leiderman 
\orcidlink{0000-0002-2257-1635}
}
\address{Faculty of Mathematics and Informatics, A. Mickiewicz University,
61-614 Pozna\'{n}, Poland and Institute of Mathematics Czech Academy of Sciences, Prague, Czech Republic}
\email{kakol@amu.edu.pl}

\address{Department of Mathematics, Ben-Gurion University of the Negev, Beer Sheva, P.O.B. 653, Israel}
\email{arkady@math.bgu.ac.il}
\keywords{Asplund spaces, Weak Asplund spaces, Fr\'echet spaces, free locally convex spaces, strict projective limits}
\subjclass[2010]{Primary 46A04, 46A13, Secondary 54E52, 26B05}

\begin{abstract}
In paper \cite{Kakol-Leiderman-2} we systematized several known equivalent definitions of Fr\'echet (G\^ ateaux) Differentiability Spaces and Asplund (weak Asplund) Spaces.  As an application, we extended the classical Mazur's theorem, and also proved that the product of any family of Banach spaces $(E_{\alpha})$ is an Asplund lcs if and only if each $E_{\alpha}$ is Asplund. The actual  work continues this line of research in the frame of locally convex spaces, including  the classes of Fr\'echet spaces (i.e. metrizable and complete locally convex spaces) and  projective limits, quojections, $(LB)$-spaces and $(LF)$-spaces, as well as, the class of free locally convex spaces $L(X)$ over Tychonoff spaces $X$.

 First we prove some "negative" results:
We show that for every infinite Tychonoff space $X$ the space $L(X)$ is not even a G\^ ateaux Differentiability Space (GDS in short) and contains no infinite-dimensional Baire vector subspaces. 
 On the other hand, we show that all barrelled GDS spaces are quasi-Baire spaces, what implies  that  strict $(LF)$-spaces  are not GDS.  This fact refers, for example, to concrete  important spaces 
 $D^{m}(\Omega)$,  $D(\Omega)$, $D(\mathbb{R}^{\omega})$.

 A special role  of the space $\varphi$, i.e. an $\aleph_{0}$-dimensional vector space equipped with the finest locally convex topology, in this line of research has been distinguished and analysed.  It  seems that little is known about  the Asplund property for Fr\'echet spaces. We show however that a quojection $E$,
 i.e. a Fr\'echet space which is a strict projective limit of the corresponding Banach spaces $E_n$,  is an Asplund (weak Asplund) space if and only if each
  Banach space $E_n$ is Asplund (weak Asplund). In particular, every reflexive quojection is Asplund. 
 Some applications and several illustrating examples are provided.
% Some applications and several examples illustrating obtained results  are provided.

\end{abstract}
%%%%%%%%%%%%%%%%%%%%%%%%%%%%%%%%%%%%%%%%%%

\maketitle
\section{Introduction}\label{s:intro}
%%%%%%%%%%%%%%%%%%%%%%%%%%%%%%%%%%%%%%%%%
\bigskip
All topological spaces $X$ are assumed to be Tychonoff and all topological vector spaces  are Hausdorff.
All vector spaces are considered over the field $\mathbb R$ of real numbers.
The abbreviation  lcs means a locally convex space. For a lcs $E$ by $E_w$ and $E'_{w^{*}}$ we mean $E$ endowed with its weak topology $w=\sigma(E,E')$ and the dual $E'$ equipped with the weak$^{*}$ topology $w^{*}=\sigma(E',E)$, respectively.  By $C_p(X)$ and $C_k(X)$ we denote the space of continuous real-valued functions on $X$ endowed with the pointwise  and the compact-open topology, respectively. If $X$ is a compact space, we denote the space $C_k(X)$ by $C(X)$.

Let us recall that, as Phelps mentioned in his  book \cite{Phelps},
the first infinite-dimensional result about differentiability properties of continuous convex functions is due to  Mazur \cite{Mazur}:
{\em A continuous convex function $f:D \to \R$ defined on an open convex subset $D$ of a separable Banach space $E$, is G\^ ateaux differentiable on a dense $G_{\delta}$ subset of $D$.}

Asplund \cite{Asplund} extended Mazur's theorem in two directions. He showed that the same statement is valid for a more general class of Banach spaces;
and also he introduced a more restricted class of Banach spaces (now called Asplund spaces) in which a
stronger conclusion of Fr\'echet differentiability holds.

Similarly to Banach spaces, lcs can be classified according
to the differentiability properties of the specified class of real-valued continuous convex functions.
 We follow the abbreviations suggested in \cite{Sharp}, see also \cite{Asplund}, \cite{Fabian}, \cite{Fabian1}, \cite{Eyland_Sharp1} and  references.

\begin{definition}\label{def3}
A lcs $E$ is called Asplund (weak Asplund) if every continuous convex function $f: D \to \R$, where $D\subset E$ is a nonempty open and convex set,
is Fr\'echet (G\^ ateaux, respectively) differentiable on a dense $G_{\delta}$ subset of $D$.
Asplund (weak Asplund) spaces are abbreviated by ASP (WASP, respectively).
\end{definition}

\begin{definition}\label{def4}
A lcs $E$ is called Fr\'echet (G\^ ateaux) Differentiability Space if every continuous convex function $f: D \to \R$, where  $D\subset E$ is a nonempty open and convex set,
is Fr\'echet (G\^ ateaux, respectively) differentiable on a dense subset of $D$.
Fr\'echet (G\^ ateaux) differentiable spaces are abbreviated by FDS (GDS, respectively).
\end{definition}

It is known that for  Banach spaces the points of  Fr\'echet differentiability of any continuous convex function form a $G_{\delta}$ set in its domain; hence 
the classes of Banach ASP and Banach FDS spaces coincide.
For a Banach space $C(K)$, where $K$ is the {\it double arrows} compact space $K$,
the set of G\^ ateaux differentiability points of the $\sup$-norm is dense, but it does not contain a dense $G_{\delta}$ set (see \cite{Coban}).
There is an example of a Banach space that is GDS but not weak Asplund \cite{Moors-2}.

The class of Banach spaces which are weak Asplund  contains for instance all Weakly Compactly Generated spaces (WCG spaces, in short).
Moreover,  if a Banach space $E$ is GDS, then the  closed unit ball in $E'$ is $w^{*}$-sequentially compact, see \cite[Theorem 2.1.2]{Fabian}.
A detailed study of weak Asplund Banach spaces and their subclasses is presented in \cite{Fabian}.

Recall that in the class of Banach spaces  finite products, continuous linear images, and closed subspaces of  Asplund spaces are Asplund. 
Every closed subspace of a WCG Banach space is weak Asplund (see \cite{Fabian}).
The largest known subclass of Banach Asplund spaces with reasonable stability
properties was introduced by Stegall \cite{Stegall1}.

The classification of lcs according to the dense or generic differentiability
of convex functions continues the works of Asplund \cite{Asplund}, Larman and Phelps \cite{Larman} and Namioka and Phelps \cite{Namioka}.
A systematic research of lcs that are ASP, WASP, FDS, and GDS was originated in 1990 by  Sharp \cite{Sharp}, 
and then it was continued  in joint papers \cite{Eyland_Sharp1}, \cite{Eyland_Sharp2}, \cite{Eyland_Sharp3}.

Any Banach space $E$, as any completely metrizable space, is {\it Baire}, i.e. the intersection of every countable sequence of dense open subsets in $E$ is dense.
Among other results,  Sharp \cite{Sharp} obtained a generalization of the Mazur's Theorem.

\begin{theorem}\label{theorem_sharp}\cite[Theorem 2.1]{Sharp}
 All separable Baire topological linear spaces (not necessarily lcs) are WASP.
\end{theorem}

Several authors much later rediscovered this result without any proper citation of the original publication \cite{Sharp};
the list of such papers includes \cite{Corbacho}, \cite{Zheng}, \cite{Lee}.
It appears that other important works \cite{Eyland_Sharp1}, \cite{Eyland_Sharp2}, \cite{Eyland_Sharp3} have also not received due attention in the later publications.

Very recently, in paper \cite{Kakol-Leiderman-2} we  systematized several known equivalent definitions of Fr\'echet (G\^ ateaux) differentiability spaces and Asplund (weak Asplund) spaces.
In fact we observed in \cite[Proposition 1.9]{Kakol-Leiderman-2} the following

\begin{proposition}\label{prop_eq4} Let $E$ be a Baire lcs.\hfill
\begin{enumerate}
\item[{\rm (1)}] $E$ is WASP if and only if every continuous convex function $f: E \to \R$ is G\^ ateaux differentiable on a dense $G_{\delta}$ subset of $E$.
\item[{\rm (2)}] $E$ is ASP if and only if every continuous convex function $f: E \to \R$ is Fr\'echet differentiable on a dense $G_{\delta}$ subset of $E$.
\end{enumerate}
\end{proposition}

As an application we extended Sharp's Theorem \ref{theorem_sharp}, see \cite[Theorem 2.2]{Kakol-Leiderman-2}.

\begin{theorem}[K\c akol-Leiderman]\label{pro-2}
Let $E$ be a separable Baire lcs and let $Y$ be the product $\prod_{\alpha\in A} E_{\alpha}$
of any family of separable Fr\'echet spaces. Then the product $E \times Y$ is weak Asplund.
\end{theorem}

Also, we proved that the product of any family of Banach spaces $(E_{\alpha})$ is an Asplund lcs if and only if each $E_{\alpha}$ is Asplund, see Theorem \ref{product}.

Let us recall  the following remarkable  results describing  Banach spaces which are Asplund; for other  references and results  see \cite{Fabian1}.
 \begin{theorem}[Deville--Godefroy--Zizler, Asplund, Namioka--Phelps]\label{separable-dual}
For a  Banach space $E$ the following assertions are equivalent:
\begin{enumerate}
\item $E$ is an Asplund space.
\item Every separable Banach subspace of $E$ has separable dual.
\item Every closed subspace of $E$ is Asplund.
\item Every equivalent norm on $E$ is Fr\'echet differentiable at some point of $E$.
\item The dual of $E$ has the Radon-Nikodym Property.
\end{enumerate}
\end{theorem}

For Banach spaces $C(X)$ we note the following  result due to Namioka and Phelps \cite{Namioka}. Recall first that a topological space $X$ is called {\it scattered} if every non-empty subset $A$ of $X$ has an isolated point in $A$.

\begin{theorem}[Namioka--Phelps]\label{NP}
A compact space $X$ is scattered if and only if  $C(X)$ is Asplund.
\end{theorem}

In the current paper we continue this line of research in the frame of lcs and, in particular,  Fr\'echet lcs, i.e. metrizable and complete lcs. 
Sections \ref{GDS} and \ref{Baire} are devoted to the free locally convex spaces.
Free topological vector spaces were introduced by  Markov \cite{Mar}, without any details.
 Later Raikov constructed the free locally convex space $L(X)$ for a uniform space $X$ \cite{Raikov}.
Free lcs constitute a very important subclass of locally convex spaces.
It suffices to mention that every lcs is a linear continuous quotient of a free lcs
and that every Tychonoff space $X$ embeds as a {\it closed} subspace into $L(X)$.
Also, every $k$-space $X$ admits a homeomorphic embedding of $L(X)$ into the double function space $C_k(C_k(X))$ \cite{Usp08}.

 Motivated by Theorem \ref{weak} below we prove  that for an infinite Tychonoff space $X$ the free locally convex space $L(X)$ is not even a GDS space (although $L(X)_w$  is Asplund), 
see Theorem \ref{Theorem1} and  Corollary \ref{consequence}. We show also that for any $X$, the lcs $L(X)$ contains no infinite-dimensional Fr\'echet locally convex subspace, 
see Corollary \ref{corollary1}.  Hence, every Banach subspace of $L(X)$ has separable dual but $L(X)$ is not GDS provided $X$ is infinite.

In order to prove our Theorem \ref{Theorem1} about $L(X)$, in Section \ref{x} first we show that $\varphi$ is not a GDS space, where $\varphi$ means an $\aleph_0$-dimensional vector space endowed with the finest locally convex topology.  Since the GDS property is inherited by continuous linear surjections, the latter facts apply to show that each infinite-dimensional lcs $E$ containing a complemented copy of the space $\varphi$ is not a GDS space, therefore $E$  is not weak Asplund. Consequently,  the  barrelled lcs which are not quasi-Baire cannot be  GDS in view of  Narayanaswami-Saxon's  \cite[Theorem 1]{Na-Saxon}, see Theorem \ref{complemented} below. Hence  strict $(LF)$-spaces (as containing complemented copies of $\varphi$) are not GDS, see Corollary \ref{LF}.  This fact refers, for example, to concrete  important spaces  $D^{m}(\Omega)$ (strict $(LB)$-space),  $D(\Omega)$, $D(\mathbb{R}^{\omega})$  (strict $(LF)$-spaces), see \cite{Bierstedt} for definitions and an analysis of these spaces.

We will show, as a direct application of the above facts around the space $\varphi$,  that the space $C_k(Q)$ (over the space of rationals $Q$) is an Asplund space (by Theorem \ref{FFF}) which contains  a closed copy of $\varphi$, see Corollary \ref{closed}. This  provides a simple  example showing that Asplund spaces may contain  closed subspaces  which are even not GDS, what  answers in the negative  Problem \ref{P1} (2 and 3)  gathered  in \cite[p. 170]{Fabian}.

It  seems that little is known about  the Asplund property for Fr\'echet lcs, although  it should  be mentioned that every Montel Fr\'echet lcs is an Asplund space \cite[Therorem 3.4]{Sharp}. 
Bearing in mind that each Fr\'echet lcs is the projective limit of a certain sequence of Banach spaces, the question arises about the potential use of Asplund's theory, but in a wider class of Fr\'echet lcs. The Factor Theorem from the Eyland-Sharp \cite[Theorem 2.2]{Eyland_Sharp1} turns out to be an important motivation for studying Asplund's property in the class of Fr\'echet spaces.
Nevertheless, it seems all the more surprising that the mentioned (and used) below  publications of Eyland and Sharp \cite{Eyland_Sharp1}, \cite{Sharp}, have remained rather unnoticed by other researchers.

It  is  unclear if Theorem \ref{separable-dual} remains true for Fr\'echet lcs $E$. We  show however that this is the case if $E$ is a {\it quojection}.
 Recall here that a Fr\'echet lcs $E$  is called a  {\it quojection} if the space $E$ is the strict projective limit of Banach spaces $E_n$ 
 (shortly we denote $E =s-proj_n(E_n, P_n)$), i.e.  each
continuous linear map $P_n : E \rightarrow E_n$ is surjective, see \cite[Definition 8.4.27]{PB}.

In Section \ref{s:2} we prove that a quojection $E$ with its defining sequence of Banach spaces $(E_{n})_n$ is an Asplund (weak Asplund) space if and only if each  Banach space $E_n$ is Asplund (weak Asplund), see Theorem \ref{fifth}.
Moreover, Theorem \ref{seven} asserts that if  $E$ is  a quojection which is Asplund, and  $F$ is a closed vector subspace of $E$ which is a quojection, then $F$ is Asplund.
% For example:
%\begin{enumerate}
%\item  If the strong dual of a quojection $E$ is separable, then E is Asplund, see Theorem \ref{six}.
%\item  If $E$ is an Asplund quojection space  and $F$ is a Banach  subspace of $E$, then $F$ is Asplund, see Theorem \ref{seven}.
%\end{enumerate}
Several open questions illustrating presented results are provided.
\section{preliminaries}\label{s:1}
%%%%%%%%%%%%%%%%%%%%%%%
\bigskip
The following theorem for any infinite Tychonoff space $X$ is proved in \cite{GKKM}.
It essentially supplements Theorem \ref{FFF} below.
\begin{theorem}[Gabriyelyan-K\c akol-Kubi\'s-Marciszewski]\label{xx}
For a Tychonoff space $X$ the following assertions are equivalent:
\begin{enumerate}
\item Every compact subset of $X$ is scattered.
\item Every Banach subspace of $C_k(X)$ is Asplund.
\item Every separable Banach subspace of $C_k(X)$ has separable dual.
\item The space $C_k(X)$ does not  contain a  copy of $\ell_{1}.$
\end{enumerate}
\end{theorem}

One can ask for which  (non-compact)  spaces $X$ any of the above conditions is equivalent to the scatteredness of the whole $X$.
The answer to this question is given for several natural classes
of non-compact Tychonoff spaces, including $\omega$-bounded spaces $X$ (for details see \cite{KKL}).
Note that similar locally convex versions of other Banach spaces properties (like (NP)-property) were recently studied in \cite{KM}.

Let $E$ be a lcs with a defining family $\mathcal{P}$  of seminorms generating the topology of $E$. For each $p\in\mathcal{P}$ define the norm $\overline{p}(x+ker(p))=p(x)$, $x\in E$, on the quotient space $E/ker(p)$.  Set $E_p=(E/ker(p),\overline{p})$ for each $p\in\mathcal{P}$.

 Following \cite{Eyland_Sharp1} we call $E$ {\it bound covering}, if there exists in E a defining family $\mathcal{P}$ such that the quotient map  $\pi_{p}: E\rightarrow E_p$ is open and for each bounded set $B\subset E_p$ there exists a bounded set $A\subset E$ with $B\subset \pi_{p}(A)$.

\begin{theorem}[Eyland-Sharp]\label{factor}\hfill
\begin{itemize}
\item[{\rm (i)}] If  $E$ is  a bound covering lcs with a defining family $\mathcal{P}$ such that for  each $p\in\mathcal{P}$ the space $E_p$ is Asplund (weak Asplund), then $E$ is Asplund (weak Asplund).
\item[{\rm (ii)}] The converse also holds if $E$ is a Fr\'echet lcs.
\item[{\rm (iii)}] If $f$ is a continuous convex function defined on an open convex subset $D$ of a bound covering  lcs $E$, the set of points of Fr\'echet differentiability of $f$ is a $G_{\delta}$ set in $D$.
\end{itemize}
\end{theorem}

Theorem \ref{factor} is a consequence of  \cite[Theorem 2.2]{Eyland_Sharp1},  explicitly follows from  \cite[Theorem 3.3 (2), Theorem 2.3 (3,4), Theorem 3.2]{Eyland_Sharp1}; 
it was used to show the following characterization of the Asplund property for the spaces $C_k(X)$, see \cite[Corollary 3.8]{Eyland_Sharp1}. 

\begin{theorem}[Eyland-Sharp]\label{FFF}
Let $X$ be a Tychonoff space.
$C_k(X)$ is an Asplund space if and only if every compact subset of $X$ is scattered. 
\end{theorem}
Indeed, the only application of Theorem \ref{factor}, which Eyland and Sharp discovered in  \cite{Eyland_Sharp1},
 deals  with  the following procedure which lead finally to the proof of Theorem \ref{FFF} (see also \cite[Example 3.6]{Eyland_Sharp1}):
Let $\mathcal{K}(X)$ be the family of all compact subsets of $X$. Endow the space $C_k(X)$ with the topology generated by seminorms $$p_{K}(f)=sup_{x\in K}|f(x)|,\,\,K\in\mathcal{K}(X).$$
The quotient space $C_k(X)/ker(p_{K})$ is isomorphic to the Banach space $(C(K),p_K)$ 
(via the isomorphism  defined by $(f+ker(p_K)) \mapsto f|K$) and the quotient map $\pi_{K}: C_k(X)\rightarrow (C(K),p_K)$ is bound covering,
i.e. if $B$ is the  unit ball in the Banach space  $C(K)$, then $B\subset \pi_{K}(M)$ for the  bounded  set $$M=\{f\in C(X): |f(x)|<1, x\in X\}\subset C_k(X).$$
The family $\{\pi_{K}, (C(K),p_K): K\in\mathcal{K}(X)\}$ was used in  the proof of Theorem \ref{FFF} to conclude that  for a Tychonoff space $X$ the space  $C_k(X)$ is Asplund if and only if  for each compact $K\subset X$ the Banach space  $C(K)$ is Asplund.

Note however that Theorem \ref{factor} can be also applied to prove the next Sharp's result \cite[Theorem 5.5]{Sharp}.
\begin{theorem}[Sharp]\label{weak}
A lcs $E_w$ equipped with the weak topology $w=\sigma(E,E')$  is an Asplund space. Consequently, for each Tychonoff space $X$ the space $C_p(X)$ is an Asplund space.
\end{theorem}
\begin{proof}
Let  $E'$ be  the topological dual of $E$ and $p^{*}(x)=|x^{*}(x)|$, $x\in E$, $x^{*}\in E'$. Then the family of seminorms $p^{*}$ generates the weak topology of $E_w$. It is well known that $E=ker(p^{*})\oplus Z$ (topologically), and  dim $Z=1$. If $Q: E\rightarrow E/ker(p^{*})$ is the quotient map, then $Q|Z: Z\rightarrow E/ker(p^{*})$ is an isomorphism, so $Q$ is a bound covering map, i.e. for each bounded set $B\subset E/ker(p^{*})$ there exists a bounded set $A$  in $E$ such that $Q(A)\subset B.$ Moreover, the space $(E/ker(p^{*}), \overline{p^{*}})$ is a Banach space isomorphic to the quotient space $E/ker(p^{*})$, where
$$\overline{p^{*}}(x+ker(p^{*}))= p^{*}(x)=|x^{*}(x)|,\, x\in E.$$ Hence the family of spaces $E_{p^{*}}=(E/ker(p^{*}), \overline{p^{*}})$, for $x^{*}\in E'$,  satisfies  the assumptions of Theorem \ref{factor} (i)  what implies that $E_w$ is Asplund.
\end{proof}
%\begin{problem}
%Assume $X$ is compact and scattered. Let $\tau_p$ and $\tau_{\infty}$ be the pointwise and the uniform topology on $C(X)$, respectively. Does there exist a locally convex topology $\xi$ on $C(X)$ such that $\tau_{p}< \xi %<\tau_{\infty}$ and $\xi$ lacks  the Asplund property?
%\end{problem}
%Note however that every Banach subspace $F$ of $(C(X),\xi)$ is Asplund since $\xi|F=\tau_{\infty}|F$ and Theorem \ref{separable-dual} applies.
\section{Weak Asplund locally convex spaces and Quasi-Baire spaces}\label{x}
%%%%%%%%%%%%%%%%%%%%%%%%%%%%
\bigskip
For Banach spaces  several questions around  weak Asplund spaces are still open. Recall the following two fundamental  problems.
\begin{problem}\label{P1}
For a Banach space $E$ consider the following two questions, \cite[p. 170]{Fabian}.
\begin{enumerate}
%\item If E is weak Asplund, is then $E\times\mathbb{R}$ weak Asplund?
\item If $E$ is weak Asplund, is any (uncomplemented) subspace of $E$  weak Asplund?
\item If $E$  is GDS, is any (uncomplemented) subspace of it also GDS?
\end{enumerate}
\end{problem}
It is known  that every complemented subspace of a (weak) Asplund Fr\'echet lcs  is weak Asplund, see \cite[Theorem 1.3]{Eyland_Sharp1}.
\begin{problem} Is every closed subspace of an Asplund Fr\'echet lcs also Asplund?
\end{problem}
As we have mentioned above, Theorem \ref{pro-2} extended  Sharp's Theorem \ref{theorem_sharp}.
For spaces $C_k(X)$ we have even the following
\begin{corollary} \label{third}
Every separable space  $C_k(X)$ is a weak Asplund space.
\end{corollary}
\begin{proof}
If $K$ is any compact subset of $X$ then the Banach space $C(K)$ is separable as a continuous image of $C_k(X)$ under the restriction mapping.
So $C(K)$ must be weak Asplund. Applying  Theorem \ref{factor} (i) with the help of the comment after Theorem \ref{FFF}   we derive that $C_k(X)$ is a weak Asplund space.
\end{proof}
\begin{example}\label{first}
A separable Banach space $C([0,1])$ is weak Asplund but not Asplund by Theorem \ref{NP} because $[0,1]$ is not scattered.
\end{example}

\begin{example}\label{second}
Let $M$ be a metric space.  For a fixed point $e\in M$, $Lip_0(M)$ denotes the Banach space of all real Lipschitz functions  on $M$ such that $f(e)=0$, equipped with the norm $\|f\|_{lip_0(M)}=lip(f)$, where $lip(f)$ denotes the Lipschitz constant of $f$.  By $\mathcal{F}(M)$ we denote the predual of $Lip_0(M)$, i.e. $\mathcal{F}(M)^{'}=Lip_0(M)$. Recall that (*) $$ \mathcal{F}(M) = \overline{span}\{\delta_{x}: x\in M\}\subset Lip_0(M)^{''},$$ and
$\mathcal{F}(M)$ contains a complemented copy of $\ell_{1}(d(M))$, \cite[Proposition 3]{Hajek} and  $Lip_0(M)$ contains a copy of $\ell_{\infty}(d(M))$ \cite[Theorem 8]{Doucha}. Therefore,

(1)  $\mathcal{F}(M)$ is  weak Asplund if and only if $M$ is separable. Indeed, if $M$ is separable, we apply (*) to derive that $\mathcal{F}(M)$ is separable; hence it is weak Asplund. Conversely, if
 $\mathcal{F}(M)$ is weak Asplund then $\ell_{1}(d(M))$ is weak Asplund (since  $\mathcal{F}(M)$ is mapped onto $\ell_{1}(d(M))$ by a continuous linear projection and  \cite[Theorem 1.3 (1)]{Eyland_Sharp1} applies.
Note however that $\ell_{1}(\Gamma)$ is not weak Asplund for uncountable sets $\Gamma$, see \cite[Ex. 8.26]{Fabian1}.

(2) (CH)  $\mathcal{F}(M)$ is weak Asplund if and only if the closed unit ball in $\mathcal{F}(M)'$ is $w^{*}$-sequentially compact. Indeed, if $\mathcal{F}(M)$ is weak Asplund, then the conclusion holds by  \cite[Theorem 2.1.2]{Fabian}. For the converse we apply item (1) and \cite[p. 226]{Diestel} stating that any Banach space containing a copy of $\ell_{1}(\mathfrak{c})$ cannot have $w^{*}$-sequentially compact dual ball.

(3) For infinite $M$ the space $\mathcal{F}(M)$ is not Asplund (since contains a copy of $\ell_{1}$ and Theorem \ref{separable-dual} applies).

(4) The space $Lip_0(M)$ over an infinite metric space $M$ is not Asplund since it contains a copy of $\ell_{\infty}$ (hence contains also $\ell_{1}$).
\end{example}

For the spaces $C_k(X)$ we reiterate the following question which is  unsolved even for compact spaces $X$, see \cite[2.4 Notes and Remarks]{Fabian}.
\begin{problem}
Characterize in terms of $X$ the weak Asplund property for $C_k(X)$.
\end{problem}
In order to prove  the main result of Section \ref{GDS} we  will need  the following useful facts related with the space $\varphi$.
\begin{fact} \label{fact1} Let $E$ and $F$  be  lcs and $T:F \to E$ be a linear continuous operator onto.
If $F$ is GDS then so is $E$.
\end{fact}
\begin{proof} This is provided in \cite[Theorem 1.1]{Eyland_Sharp1}.
\end{proof}
\begin{definition} Denote by $\ell_0$ the linear space of countable real sequences with only finitely many nonzero elements.
We consider $\ell_0$ as a normed subspace of the Banach space $\ell_1$.
Denote the natural Hamel basis in $\ell_0$ by $\{e_n: n\in \N\}$. Each $e_n$ is a sequence with $1$ on the $n$-th place, on all other places we put zeros.
By $\mO$ we denote the zero element of $\ell_0$.
\end{definition}
\begin{fact} \label{fact2}
$f(x) = ||x||$ is a continuous convex function defined on $\ell_0$ which is not G\^ ateaux differentiable at any point.
\end{fact}
\begin{proof} This is included in \cite[Example 1.4(b)]{Phelps}.
\end{proof}
By $\varphi$ we denote an $\aleph_{0}$-dimensional vector space endowed with the finest locally convex topology.  It is obvious that $\varphi$ is the countable inductive limit of finite-dimensional Banach subspaces; hence  $\varphi$  is a Montel $(DF)$-space and is  a nonmetrizable $(LB)$-space,  see \cite{KKP} for several facts related with this space. Next Fact \ref{fact4}  is also mentioned in  \cite[Example 6.1]{Sharp} but without using the name $\varphi$.
\begin{fact} \label{fact4}
The space $\varphi$ is not GDS. Indeed, let $I: \varphi\rightarrow \ell_0$ be the (continuous) identity. If $\varphi$ was  GDS, then $\ell_0$ would also be  GDS, a contradiction by Fact \ref{fact2}.
\end{fact}

 Let us recall a few Baire type conditions for a  lcs $E$. Evidently, a topological space is Baire, if it is not the union of an increasing sequence of nowhere dense subsets.

\begin{enumerate}
\item $E$ is called a Baire-like space, if $E$ is not the union  of an increasing sequence of nowhere dense absolutely convex subsets.
\item $E$ is a barrelled space, if every $\sigma(E',E)$-bounded set in $E'$ is equicontinuous.
\item $E$ is a quasi-Baire space, if $E$ is barrelled and $E$ is not the union of an increasing sequence of nowhere-dense linear subspaces of $E$.
%\item Baire $\Rightarrow$ Baire-like $\Rightarrow$ quasi-Baire $\Rightarrow$ barrelled.
\end{enumerate}

Clearly,  Baire $\Rightarrow$ Baire-like $\Rightarrow$ quasi-Baire $\Rightarrow$ barrelled.
The following result of Narayanaswami-Saxon \cite[Theorem 1]{Na-Saxon} will be used below.

\begin{theorem}[Narayanaswami-Saxon]\label{complemented}
Let $E$ be a barrelled space. Then $E$ is not quasi-Baire if and only if $E$ contains a complemented copy of $\varphi$.
\end{theorem}
This theorem and Fact \ref{fact1} combined with Fact \ref{fact4} imply the next
\begin{corollary} \label{quasi}
Every barrelled GDS space is a quasi-Baire space.
\end{corollary}
Applying \cite[Theorem 4]{Kakol-Saxon-Todd}, which states that no space $C_k(X)$ can be covered by an increasing sequence of closed proper vector subspaces, we derive a stronger fact for spaces $C_k(X)$.
\begin{corollary}\label{quasi1}
Every barrelled space $C_k(X)$ is quasi-Baire.
\end{corollary}
Next corollary uses both Facts \ref{fact1} and \ref{fact4} to provide a necessary condition for barrelled lcs to be GDS.
\begin{corollary}\label{fourth}
Let $E$ be an infinite-dimensional   barrelled space and let $\{x_j:j\in B\}$ be a Hamel basis of $E$, and  let $\{(x_j, f_j): j\in B\}$ be  the corresponding
biorthogonal system. If $E$ is  GDS, then the set $C=\{j\in B: f_j\in E'\}$ is finite.
\end{corollary}
\begin{proof}
Let $D=\{j\in B: f_j\notin E'\}$. Then $lin\{x_j: j\in C\}\oplus lin\{x_j: j\in D\} = E$ (topologically), and $F=lin\{x_j: j\in C\}$ carries  the finest locally  convex  topology \cite[Theorem 2]{Saxon-Robertson}. If $C$  is infinite, then $F$  contains a copy  of $\varphi$. On the other hand, every vector subspace of a lcs $F$ carrying the finest locally convex topology is complemented in $F$ (since any linear surjection between lcs  endowed with  the finest  locally convex topologies is continuous).  Hence $\varphi$ would be a  complemented subspace  in the whole space $E$, what is impossible  provided $E$ is GDS (by applying Facts \ref{fact1} and \ref{fact4}). Hence $C$ must be finite.
\end{proof}
Last Corollary \ref{fourth} combined with Theorem \ref{FFF} implies next

\begin{corollary}
Let $X$ be a metrizable space whose every compact subset  is  scattered. Let $\{x_j:j\in B\}$ be a Hamel basis of $C_k(X)$, and  let $\{(x_j, f_j): j\in B\}$ be  the corresponding
biorthogonal system. Then  $\{j\in B: f_j\in E'\}$ is finite.
\end{corollary}
Recall  the following concepts  which will be used in the sequel, see \cite{Na-Saxon} and \cite{PB}.
\begin{enumerate}
\item Let $(E_{n},\tau_{n})_{n}$ be an increasing  sequence of lcs  such that $\tau_{n+1}|E_n\leq \tau_{n}$ for all $n\in\omega.$
\item If $E=\bigcup_{n}E_n$, then $E$ with  the finest l.c topology $\tau$ on $E$ such that $\tau|E_n\leq\tau_n$ for all $n\in\omega$ is called the inductive limit of the above sequence.
\item If  each $(E_n,\tau_n)$ is a Fr\'echet (Banach) space,  $E_{n}\varsubsetneq E_{n+1}$, the  limit is called an $(LF)$-space ($(LB)$-space).
\item If further, for each $n$ we have $\tau_{n+1}|E_n=\tau_n$,  the limit is called strict.
\item No strict $(LB)$-space is metrizable.
\item $E$ is an  $(LF)_{2}$-space if $E$ is a non-metrizable $(LF)$-space whose each member from its defining sequence is a dense subspace of $E$.
\end{enumerate}
Recall that an $(LF)$-space $E$ is an $(LF)_{2}$-space if and only if $E$ is quasi-Baire but not Baire-like, see \cite[Theorem 4]{Na-Saxon}.

Note that the spaces $D^{m}(\Omega)$,  $D(\Omega)$ and  $D(\mathbb{R}^{\omega})$ are strict $(LF)$-spaces, see \cite{Bierstedt}.  Theorem \ref{complemented} with Fact \ref{fact4} yield  a large class of lcs without GDS.
\begin{corollary}\label{LF}
No infinite-dimensional strict $(LF)$-space is GDS.
\end{corollary}
Recall that a topological space $X$ is a $\mu$-space if the closure of every functionally bounded subset $A$  of $X$ is compact, where $A$ is functionally bounded if every continuous real-valued function on $X$ is bounded on $A$, see \cite[Definition II.6.1]{Schmets} or \cite[Definition 10.1.18]{PB}.
\begin{corollary}\label{P2}
If $X$ is metrizable, then $C_k(X)$ does not contain a complemented copy of $\varphi$. If $X$ is metrizable but  not locally compact,  $C_k(X)$ contains a copy of $\varphi$.
\end{corollary}
\begin{proof}
Since each metrizable space $X$ is a $\mu$-space,  $C_k(X)$ is barrelled, see \cite[Theorem 10.1.20]{PB}. Then   $C_k(X)$ is quasi-Baire by Corollary \ref{quasi1}. Now we apply Theorem \ref{complemented} to conclude that $C_k(X)$ does not contain a complemented copy of $\varphi$. On the other hand, if additionally $X$ is not locally compact, $C_k(X)$ is not Baire-like, see \cite[Proposition 2.19]{KKP}. Hence,  by \cite[Corollary 2.4]{KKP},   $C_k(X)$ contains a copy of $\varphi$.
\end{proof}
Last Corollary \ref{P2} yields  the following next consequence which  provides a negative answer on  Problem \ref{P1} (both 2 and 3) for non-Banach spaces $E$. 

\begin{corollary}\label{closed}
The space $C_{k}(Q)$ contains a copy  of $\varphi$  but does not contain any complemented copy of $\varphi$. Hence the Aspund space $C_k(Q)$ contains a closed vector subspace (isomorphic to $\varphi$) which is  not GDS.
\end{corollary}
The space $C_k(J)$ over irrationals $J$ is not Asplund (use Theorem \ref{FFF}) but is weak Asplund, since $C_k(J)$ is separable and Corollary \ref{third} applies. Moreover,  $C_k(J)$  contains a copy of $\varphi$ by Corollary \ref{P2}.

Since every separable Baire space is weak Asplund and  every barrelled weak Asplund lcs is quasi-Baire (see  Corollary \ref{quasi}), and  every separable $C_k(X)$ is weak Asplund (see Corollary \ref{third}), the following question may arise.
\begin{problem}
Is every separable quasi-Baire space  weak Asplund?
\end{problem}

\section{Free locally convex space $L(X)$ is not GDS} \label{GDS}
%%%%%%%%%%%%%%%%%
\bigskip
\begin{definition}
Let $X$ be a Tychonoff space. The {\em  free lcs} $L(X)$ on  $X$ is a pair consisting of a lcs $L(X)$ and  a continuous mapping $i: X\to L(X)$ such that
 every  continuous mapping $f$ from $X$ to a lcs $E$ gives rise to a unique continuous linear operator $\widehat f: L(X) \to E$  with $f=\widehat f \circ i$.
\end{definition}

The free lcs $L(X)$  always exists and is  unique. The set $X$ forms a Hamel basis for $L(X)$, and  the mapping $i: X \rightarrow L(X)$, $x\mapsto \delta_{x}$, is a topological embedding.
For simplicity, we will identify $X$ with $i(X)$.

Recall that $L(X)_w$ is Asplund for any $X$ by Theorem \ref{weak}.
We are at the position to prove the following general statement for any $L(X)$, one of the main results of our paper.
For every subspace $Y$ of $X$ we denote by $L(Y, X)$ the linear subspace of $L(X)$ generated by all elements of $Y$.
In general $L(Y, X)$ is not topologically isomorphic to $L(Y)$. Isomorphism holds if and only if every continuous pseudometric on $Y$ extends to
a continuous pseudometric on $X$ \cite{Uspenskii}.

\begin{theorem} \label{Theorem1}
Let $Y$ be an infinite subset of a Tychonoff space $X$. Then $L(Y, X)$ is not a GDS space.
In particular, $L(X)$ is not a GDS space for any infinite Tychonoff space $X$.
\end{theorem}
\begin{proof}
Since $X$ is Hausdorff and $Y\subset X$ is infinite, there is a sequence of disjoint open sets $\{U_n: n\in\omega\}$ in $X$
such that $U_n\cap Y \neq \emptyset$ for each $n\in \omega$.
Pick a point $y_n \in U_n\cap Y$ and fix continuous functions $\alpha_n: X \to [0, n^{-1}]$ such that $\supp(\alpha_n) \subset U_n$ and $\alpha_n (y_n) = n^{-1}$.
Now we define a continuous mapping $f: X \to \ell_0$ as follows:

\begin{equation}
f(x)=
    \begin{cases}
        \alpha_n(x)\,e_n & \text{, if } x \in U_n\\
        \mO & \text{, if } x \in X \setminus \bigcup_{n\in\omega} U_n
    \end{cases}
\end{equation}

It is clear that $f$ is continuous at every point $x \in U_n$.
We show that $f$ is continuous at every point $x_0 \in X \setminus \bigcup_{n\in\omega} U_n$. By definition, $f(x_0)= \mO$. Let $\epsilon > 0$.
There is $n\in \omega$ such that $n^{-1} < \epsilon$. Denote by $V_i = \{x\in X : \alpha_i (x) < \epsilon \}$. Then $V= \bigcap_{i < n} V_i$ is open, $x_0\in V$ and $||f(x)|| < \epsilon$
for every $x\in V$.

A continuous mapping $f$ by definition can be extended to a linear continuous operator
 $\widehat f$ from $L(X)$ to $\ell_0$.
It is easy to see that $\widehat f$ maps $L(Y, X)$ onto $\ell_0$, hence $L(Y, X)$ is not a GDS, by Fact \ref{fact1} and Fact \ref{fact2}.
\end{proof}

The next question arises naturally.
\begin{problem}\label{L}
Let $X$ be an infinite Tychonoff space. Does $L(X)$ contain an infinite-dimensional GDS subspace?
\end{problem}

It is known that the free lcs  $L(\omega)$ (where $\omega$ is endowed with the discrete topology) is isomorphic to the space $\varphi$ by applying Raikov's \cite[Theorem 1']{Raikov}.

The study around  the space  $L(\omega)=\varphi$ has attracted attention of several researchers from a long time.  For example, Nyikos observed \cite{nyikos}  that each sequentially  closed subset of $L(\omega)$ is closed although the sequential closure of a subset of $\varphi$ need not be closed. The  Baire category theorem applies to show that $L(\omega)$ is not a Baire-like space,  and  a barrelled  lcs $E$ is Baire-like if $E$ does not contain a  copy of $L(\omega)$, see \cite[Corollary 2.4]{KKP}.

Although $L(\omega)$ is not Fr\'echet-Urysohn, it provides some extra properties since all vector subspaces in $L(\omega)$ are closed.
Let us mention only, that in \cite{kakol-saxon} K\c akol and Saxon introduced  the property for a lcs $E$ (under the name $C_{3}^{-}$) stating that the sequential closure of every linear subspace of $E$ is sequentially closed, and we proved
\cite[Corollary 6.4]{kakol-saxon} that the only infinite-dimensional Montel (DF)-space with property  $C_{3}^{-}$ is $L(\omega)$. This implies that  barrelled $(DF)$-spaces and $(LF)$-spaces satisfying  property $C_{3}^{-}$ are exactly  of the form $M$, $L(\omega)$, or
$M\times L(\omega)$ where $M$ is metrizable, \cite[Theorems 6.11, 6.13]{kakol-saxon}.

The definition of $L(X)$ implies that the dual space $L(X)'$ of $L(X)$ is linearly isomorphic to the space $C(X)$.
 In fact, every $f\in C(X)$ can be uniquely extended from $X$ to $L(X)$ as follows: $$\widehat f(\xi)=a_{1}f(x_{1})+\dots + a_{n}f(x_{n})\,\,\, \text{for} \,\,\,\xi=a_1x_1+\dots +a_nx_n\in L(X).$$ Note also that the inverse operator $H:L(X)'\rightarrow C(X)$ to the extension operator  $f\mapsto \widehat f$ is  the restriction operator $H: \xi\mapsto \xi\restriction_X,$ where $\xi\in L(X)'$.  Via the pairing $(L(X)',L(X))=(C(X),L(X))$ we observe that
$$C_p(X)'_{w^{*}}=L(X)_{w}, \,\,\,  C_p(X)=L(X)'_{w^{*}},$$ where $E'_{w^{*}}=(E',\sigma(E',E))$.

Further, we study the strong dual $L(X)'_{\beta}$ of $L(X)$.
Following  \cite[Definition II.6.1]{Schmets} by $\mu X$ we  denote the smallest $\mu$-subspace
of $\beta X$ containing $X$, see also \cite{BG}. It is known  that $\mu X$ is equal to the intersection of all $\mu$-subspaces of $\beta X$ that contain $X$.
By \cite[Proposition 3.1, Theorem 3.3]{BG}   the space  $C_k(\mu X)$ is isomorphic to the strong dual $L(X)'_{\beta}$ if and only if $L(X)'_{\beta}$ is barrelled.
 This means that the space $L(X)$ is distinguished according to the notation from \cite[8.7.1]{Kothe}. However, by  \cite[Corollary 12]{FK} for a Tychonoff space $X$ the space $C_p(X)$ is distinguished if and only if the strong dual of $C_p(X)$ is barrelled if and only if the strong dual of $C_p(X)$ carries the finest locally convex topology.  On the other hand, for each Tychonoff space $X$ the strong dual of $C_p(X)$ lucks infinite-dimensional bounded sets, i.e bounded sets $B$ such that linear span of $B$ is an infinite-dimensional vector space, see \cite[p. 392]{FKS}. It turns out that for the spaces $L(X)$ this latter  characterization fails in general.
Indeed, using Theorem \ref{FFF} we have the following
\begin{corollary}\label{consequence0}
For an infinite $\mu$-space $X$ the following  assertions  are true.
\begin{enumerate}
\item  The strong dual $L(X)'_{\beta}$ of $L(X)$ is Asplund if and only if every compact subset of $X$ is scattered.
\item  $L(X)'_{\beta}$ is a quasi-Baire space not carrying the finest locally convex topology.
\item  In particular, if $X$ is a metrizable space which is not locally compact, the space $L(X)'_{\beta}$ is quasi-Baire but is not Baire-like.
\item  If additionally $X$ is realcompact, $L(X)'_{\beta}$ contains an infinite-dimensional bounded set.
\end{enumerate}
\end{corollary}
\begin{proof}
Since  $X$ is a $\mu$-space,  $C_k(X)$ is barrelled \cite[Theorem 10.1.20]{PB}. The first part of  (1) follows from mentioned above \cite[Proposition 3.1, Theorem 3.3]{BG} and Theorem \ref{FFF}.
To check the first part of (2) we apply Corollary \ref{quasi1}  and the fact that the barrelled space $C_k(X)$ is isomorphic to $L(X)'_{\beta}$. For the proof of the second part of (2) assume that $L(X)'_{\beta}$ carries the finest locally convex topology. Then $L(X)'_{\beta}$ contains a copy $F$  of $\varphi$. But then  $F$ is complemented in $L(X)'_{\beta}$ (since any linear surjection between lcs carrying the finest locally convex topologies is continuous). This shows that $L(X)'_{\beta}$ is not quasi-Baire again by  Theorem \ref{complemented}, a contradiction with the first claim of (2).
Part  (3) follows from   (2) and \cite[Proposition 2.19]{KKP}; recall here  that $X$ is a $\mu$-space.
In order to prove (4), first recall that every  realcompact space is a $\mu$-space. Next, note that $L(X)'_{\beta}=C_k(X)$ is bornological (i.e. every locally bounded linear map from $L(X)'_{\beta}$ to any lcs is continuous). Indeed, since $X$ is realcompact, the claim follows from the Nachbin-Shirota theorem, see for example  \cite[Theorem 10.1.12]{PB}. Assume that every bounded set in $L(X)'_{\beta}$ is finite-dimensional. Then the identity map from  $L(X)'_{\beta}$ endowed with the original strong topology $\tau$ onto the $L(X)'_{\beta}$ endowed with the finest locally convex topology is continuous. Hence,
$\tau$ is the finest locally convex topology, what  contradicts  (2).
\end{proof}

Note that  according to \cite[Theorem 3.5]{BG} every bounded set in $L(X)$ is finite-dimensional if and only if each  functionally bounded set in $X$ is finite.
As a byproduct, our Theorem \ref{weak} combined with Theorem \ref{Theorem1} yields the following
\begin{corollary}\cite[Proposition 3.11]{BG}\label{consequence}
For every infinite Tychonoff space $X$ the space $L(X)$ does not carry the weak topology.
\end{corollary}

\section{Baire subspaces of $L(X)$} \label{Baire}
%%%%%%%%%%%%%%%%%%
\bigskip
It is known that the free locally convex space $L(X)$ is a Baire space if and only if $X$ is finite.
The same is true for metrizability: $L(X)$ is metrizable if and only if $X$ is finite.
Using  \cite[Theorem 5]{FKS} we know that  the space $L(X)$ is quasibarrelled if and only if $L(X)$ is barrelled if and only if $X$ is discrete. On the other hand, the space  $L(X)$
 is complete if and only if $X$ is Dieudonn\'e complete and does not have infinite compact subsets, see \cite{Uspenskii}.

 In order to prove Theorem \ref{Theorem2}  below we will use the following  elementary combinatorial statement. For the reader's convenience we provide a complete argument.
\begin{lemma}\label{lemma1} Let $B=\{\xi_i: i=1, 2,\dots, k\}$ be a subset of an Euclidean space $\R^K$ consisting of linearly independent vectors. Then some convex linear combination
of vectors of $B$ has at least $k$ non-zero coordinates in $\R^K$.
\end{lemma}
\begin{proof}
For every $\xi \in R^K$  we denote by $N(\xi)$ the set of non-zero coordinates.
 The union of all sets $C=\bigcup_{i=1}^k N(\xi_i)$ has a size at least $k$, otherwise we get a contradiction as $B$ would be linearly dependent.
Now, by choosing positive coefficients $a_1, a_2, \dots, a_k$ recursively, which grow sufficiently fast, we get a linear combination $\sum_{i=1}^k a_i \xi_i$ with all the coordinates in
 this set $C$ being non-zero. Normalizing it, i.e. dividing by the value $\sum_{i=1}^k a_i$, we obtain an element of the convex hull of $B$ having the same property.
\end{proof}
We know that a separable Baire lcs is weak Asplund. This  combined with  Theorem \ref{Theorem1} and Problem \ref{L} may suggest  Theorem \ref{Theorem2} below. First recall that
a lcs $E$  is {\it quasi-complete} if every closed bounded subset of $E$ is complete, and $L(X)$ is quasi-complete if and only if every functionally bounded set in $X$ is finite, see \cite[Theorem 3.8]{Gabriyelyan2}. 
\begin{theorem} \label{Theorem2}
Let $X$ be a Tychonoff space. Assume that $V$ is a vector subspace of $L(X)$ with an infinite set $D$ of linearly independent vectors.  Then the following holds:
\begin{enumerate}
\item The space $V$ is not a Baire space.
\item  If $V$ is dominated by a metrizable locally convex  topology then $V$ cannot be  quasi-complete.
%\item The space $L(X)$ does not admit a stronger separable Baire locally convex topology.
\end{enumerate}
\end{theorem}
\begin{proof} Let $V \subseteq L(X)$ be a linear subspace which is Baire.
We will use the following notation. Let $\xi\in L(X)$. The length $l(\xi)$ is the natural number $n$ such that
there are distinct points $x_1, x_2, \dots , x_n \in X$ and reals $\lambda_1, \lambda_2, \dots , \lambda_n \in \R\setminus \{0\}$ for which
$$\xi = \lambda_1 x_1 + \lambda_2 x_2 + \dots \lambda_n x_n.$$
If $\xi$ is the zero element of $L(X)$, then $l(\xi)=0$. Further, $L_n(X) = \{\xi \in L(X): l(\xi)\leq n \}$ for every natural $n$.
One can easily see that all $L_n(X)$ are closed subsets of $L(X)$ \cite[Proposition 0.5.16]{Arch_book}.

Denote by $V_n = V \cap L_n(X)$. Since $V$ is Baire, at least one $V_n$ contains an open subset $A = U \cap V$,
where $U$ is open in $L(X)$. But $L(X)$ is a lcs, therefore we can assume that $A$ is convex and open in $V$.

On the contrary, assume that $V$ contains infinitely many linearly independent vectors. Then $A$ also contains some infinite
set $B=\{\xi_i: i\in \omega\}$ consisting of linearly independent vectors. Remind that every vector in $B$ has the length at most $n$.
We claim that by Lemma \ref{lemma1}  some convex linear combination $\alpha= \sum_{i=1}^{n+1} t_i \xi_i$ has the length $l(\alpha) \geq n+1$ in $L(X)$.
Since $\alpha$ is in $A$ but does not belong to $L_n(X)$ we obtain a contradiction which finishes the proof.

Now we prove item (2). Assume that $V$ is a quasi-complete vector subspace of $L(X)$ and let $\gamma$ be a stronger metrizable locally convex topology on $V$.  Take any infinite sequence $(d_{n})_{n}$  in $D$ and a sequence $(t_{n})_n$ of positive elements  in reals such that $x_{n}=t_{n}d_{n}\rightarrow 0$ in the topology $\gamma$ on $V$. Set $K=\{x_{n}:n\in\omega\}\cup \{0\}$. The closed convex hull $K_{c}$ of $K$ in the original topology of $V$ is a precompact space, see \cite[20.6.2]{Kothe}.  Since $V$ is assumed to be quasi-complete, $K_{c}$ must be a complete space.  By \cite[Corollary 3.5.4]{Ja} the space $K_{c}$ is a  compact space whose linear span is an infinite-dimensional space. On the other hand, by \cite[Corollary 3.2.10]{PB} every barrel in $L(X)$ absorbs the set $K_{c}$, i.e. for every closed (in $L(X)$) absolutely convex absorbing subset $Z$ in $L(X)$ there exists a scalar $t> 0$ such that $K_{c}\subset tZ$. This means that $K_{c}$ is an infinite-dimensional {\it barrel-bounded} subset of $L(X)$ in sense of the definition of \cite[p. 6396]{BG}. This contradicts however \cite[Theorem 2.2]{BG} stating that  for every Tychonoff space $X$ in the space $L(X)$ every barrel-bounded set must be finite-dimensional.
%To check the claim (3),  let $\xi$ be a stronger topology on $L(X)$ under which $L(X)$ is separable and Baire. By Theorem \ref{theorem_sharp}  the space $(L(X),\xi)$ is weak Asplund. Therefore by Fact \ref{fact1} we deduce that the original topology of $L(X)$ satisfies GDS, a contradiction with Theorem \ref{Theorem1} (2).
\end{proof}
We do not know if the assumption  in (2) of quasi-completeness can  be removed.
\begin{corollary} \label{corollary1}
For any Tychonoff space $X$, the lcs $L(X)$ contains no infinite-dimensional Fr\'echet locally convex subspace. Consequently, every Banach subspace of $L(X)$ has separable dual but $L(X)$ is not a GDS space provided $X$ is infinite.
\end{corollary}

With the help of results from \cite{LU} and \cite{Usp08} one can show that for every compact $X$ the space 
$L(X)$ does not contain infinite-dimensional linear normed subspaces.

\begin{problem} Is it true that for any Tychonoff space $X$ the free lcs $L(X)$ does not have infinite-dimensional linear normed subspaces?
\end{problem}

%%%%%%%%%%%%%%%%%%%%%%%%%%%%%%%%%
\section{Projective limits, quojections and the Asplund property}\label{s:2}
%%%%%%%%%%%%%%%%%%%%%
\bigskip
Let us recall the fundamental concepts related with projective limits and quojections.

Let  $(E_n)_n$  be a sequence of Banach spaces. For all $m,n$, $m\geq  n$ let
$P_{nm}:E_m \rightarrow E_n$ be  a continuous linear mapping such that  $P_{nn}$ is the identity
and $$P_{nm}\circ P_{ms} = P_{ns} (s\geq m\geq n).$$
The pair $((E_n),(P_{nm})_{m\geq n})$ is called a projective sequence and
$$E:=( (x(n))\in \prod_{n}E_n : P_{nm}(x(m))=x(n),  m\geq n)$$
with the product topology is called its projective limit.
Canonical projections $P_n: E \rightarrow E_n$ are defined as follows:
$(x(m))_{m}\mapsto  x(n)$.

For two Banach spaces $E$ and $F$ the symbol $E\approx F$ means that $E$ and $F$  are isomorphic.
The following result has been obtained recently in \cite[Theorem 2.3]{Kakol-Leiderman-2}.

\begin{theorem}[K\c akol-Leiderman]\label{product}
The product of any family of Banach spaces $(E_{\alpha})$ is an Asplund lcs if and only if each $E_{\alpha}$ is Asplund.
\end{theorem}

Note  also that every Fr\'echet lcs $E$ is the projective limit of a sequence of its local Banach spaces $E_{n}$, see \cite[Definition, p. 279, Remark 24.5]{Meise-Vogt}.
\begin{problem} Is a Fr\'echet lcs $E$ an Asplund space provided every its local Banach space $E_n$ is Asplund?
\end{problem}
Let $((E_n)_n, (P_{nm})_{m\geq n})$ be a  projective sequence of Banach spaces for a Fr\'echet lcs $E$ and let  $E$ be its projective limit. We call  $E$  a {\it strict projective limit} of the sequence $(E_{n})_{n}$ if each $P_n:E\rightarrow E_n$
is surjective. Then, by the open mapping theorem, see \cite[Theorem 8.4.11]{PB},  the quotient space $E/ker(P_n)$ is isomorphic to the Banach space $E_n$, $n\in\omega$.
 The strict projective limits  of Banach spaces are called {\it quojections} \cite{Bellenot}.

It is known that a Fr\'echet lcs  $E$ is a quojection if and only if  for every continuous seminorm $p$ on $E$, the quotient space
$E/ker p$ endowed with the quotient topology is a Banach space, if and only if every quotient of $E$  which admits a continuous norm is a Banach space,  see \cite[Proposition 8.4.33]{PB}.
 The latter fact implies that
 a quotient of a quojection is again a quojection, see   \cite{Bellenot}, although the quojection property is not inherited by closed subspaces \cite[p. 238]{Metafune}. Trivially, every countable product of Banach spaces is a quojection, see \cite[Observation 8.4.28]{PB}.
 On the other hand, by \cite[Theorem]{Ostrovskii} we derive that there exist quojections not isomorphic to countable products of Banach spaces and  without infinite-dimensional Banach subspaces.
 Note however that there exist Fr\'echet lcs which are Asplund not being quojections, see Example \ref{fast}.
\begin{example}\label{fast}
The space of rapidly decreasing sequences $\mathfrak{s}=\{x=(x_n)\in \mathbb{R}^\omega: |x|_{q}=(\sum_{j=1}^{\infty}|x_j|^{2}j^{2q})^{1/2} < {\infty}, q\in N_0 \}$ with the family  of  norms $|x|_{q}$ is a nuclear (hence Montel separable) Fr\'echet space, see for example \cite{Cias} for details and several references concerning this space. Since $\mathfrak{s}$ is a Montel Fr\'echet lcs, it must be  Asplund by \cite[Therorem 3.4]{Sharp}.  Nevertheless $\mathfrak{s}$ is not  a quojection since the linking maps $P_{nm}$ being compact cannot be surjective, where all spaces  $E_{n}$ defining the  projective limit of $\mathfrak{s}$ are  Hilbert spaces with the norm $|x|_{n}$.
\end{example}

Note that if $E$ is a Fr\'echet  lcs, then $E$   is a (closed) subspace of the product $\prod_n E_{n}$ of the corresponding local Banach spaces $E_n$. By Theorem \ref{product} the space $\prod_n E_{n}$ is Asplund provided each $E_n$ is  Asplund. It is not clear if then  $E$ is Asplund.
 Nevertheless, we prove the following assertion which is one of the main results of our paper.

\begin{theorem}\label{fifth}
Let $E$ be a quojection, so $E=s-proj((E_n)_n, (P_{n})_n)$ is the strict projective limit of Banach spaces  $E_{n}$. Then $E$ is an Asplund (weak Asplund) space if and only if each space $E_n$ is an Asplund (weak Asplund) space.
\end{theorem}

\begin{proof}
Assume that each Banach space $E_n$ is Asplund (weak Asplund). Let $z_{n}(x)=\|x\|_{n}$ be the original norm of the space $E_n$. Let $P_n: E\rightarrow E_n$ be a continuous linear (open) surjection which exists by the definition of the strict projective limit. Then the countable family $\{z_n\circ P_n:n\in\omega\}$ composes a defining family of continuous seminorms on the Fr\'echet lcs $E$, see
\cite[Definition, p. 278--279]{Meise-Vogt}. Note that $ker(z_n\circ P_n)=ker(P_n)$ for each $n\in\omega$. Hence $E/ker(z_n\circ P_n)=E/ker(P_n)$  is isomorphic to the Asplund (weak Asplund) Banach space $E_{n}$.
Set $h_n=z_n\circ P_n$ for each $n\in\omega$.
Now define the norm $$\overline{h_n}(x+ker(h_n))=h_n(x)$$ on $E/ker(h_n)$ for each $x\in E$.
Since $E$ is a quojection  Fr\'echet lcs we apply \cite[Theorem 8.4.33]{PB} to derive that each  $(E/ker(h_n), \overline{h_n})$ is a Banach space (isomorphic to the Banach space $E/ker(h_{n})$ since the (continuous) identity map from  $E/ker(h_{n})$ onto $(E/ker(h_n), \overline{h_n})$ is open by the Open Mapping Theorem \cite[Theorem 1.2.36]{PB}). Set $F_n=(E/ker(h_n), \overline{h_n})$, $n\in\omega$. Let $Q_n: E \rightarrow F_n$ be the quotient map. Since  every quojection is a quasinormable space \cite[Section 1, p. 217]{Domanski}, we apply \cite[Theorem 1]{Minarro} and deduce that each $Q_n$ is a bound covering map from the quasinormable space $E$ onto the Banach space $F_n$. Therefore,  we proved that $E$ is a bound covering  Fr\'echet lcs  with the corresponding sequence  of Banach spaces  $F_n=(E/ker(h_n), \overline{h_n})$ and the  bound covering  quotient maps $Q_n: E\rightarrow F_n$. Finally by Theorem \ref{factor} (i) we conclude that $E$ is an Asplund (weak Asplund) space.

Conversely, assume that the Fr\'echet lcs $E=s-proj((E_n)_n, (P_{n})_n)$ is an Asplund (weak Asplund) space. We show that each defining Banach space $E_n$ is Asplund (weak Asplund). Again, as above, each space $F_n=(E/ker(h_n), \overline{h_n})$ is a Banach space isomorphic to the quotient space $E/ker(h_n)\approx E_{n}.$ We know also that each quotient map $$Q_n: E\rightarrow F_n (\approx E/ker(h_n))$$ is  bound covering again by \cite[Theorem 1]{Minarro}. Since $E$ is a Fr\'echet lcs, we apply Theorem \ref{factor} (ii) to conclude that each  $E_n$ is Asplund (weak Asplund).
\end{proof}

Since a quojection $E=s-proj((E_n)_n, (P_{n})_n)$ is reflexive if and only if each Banach space $E_{n}$ is reflexive, \cite[Observation 8.4.31]{PB}, the previous Theorem \ref{fifth} and Theorem \ref{separable-dual} imply
\begin{corollary}
A reflexive quojection $E=s-proj((E_n)_n, (P_{n})_n)$ is Asplund.
\end{corollary}

Recall  that for a lcs $E$ {\it the strong dual} $(E',\beta(E',E))$  of $E$ is the dual $E'$ endowed with the  locally convex topology $\beta(E',E)$ of the uniform convergence on bounded sets of the space $E$. It is also known that every separable quojection is a quotient of the space $\ell_{1}^{\omega}$ by a quojection subspace, see \cite{Metafune}.
It is known that a separable Banach space $E$  is Asplund if and only if the  dual $E'$ is separable, see Theorem \ref{separable-dual}.
We prove the following variant for Fr\'echet lcs.
\begin{theorem}\label{six}
Let $E$ be a  quojection. Then  $(E',\beta(E',E))$  is separable if and only if $E$ is a separable Asplund space.
\end{theorem}
\begin{proof}
Assume $E$ is Asplund and separable. Since  $E=s-proj((E_n)_n, (P_{n})_n)$, and $E$ is separable, the space $E_n$ is separable as well for each $n\in\omega$. Now we use the  argument from the proof of the previous  Theorem \ref{fifth}  to derive that each $E_n$ is Asplund. Hence the dual $E_n'$ with the dual norm topology is separable in view of Theorem \ref{separable-dual}. Since $E$ is the strict projective limit space of Banach spaces $E_{n}$, we apply \cite[Proposition 8.4.30]{PB} to show that the strong dual $(E',\beta(E',E))$ of $E$ is the strict inductive limit of the sequence of separable Banach spaces $E'_{n}$.  Consequently, the space $(E',\beta(E',E))$ is separable as claimed.

Conversely, assume  $(E',\beta(E',E))$ is separable. Then, as $(E',\beta(E',E))$  is the strict inductive limit of the sequence of  Banach spaces $E'_{n}$, each  $E'_{n}$ is a topological subspace of $(E',\beta(E',E))$. Since every metrizable vector subspace of a separable lcs is separable, see \cite[Corollary 2.5.5]{PB},  in the relative topology of $(E',\beta(E',E))$ each $E_{n}'$ is separable. Hence each $E_n$ is Asplund again by Theorem \ref{separable-dual}. Now it is enough to apply Theorem \ref{fifth} to derive that  $E$ is Asplund. The separability of the Fr\'echet lcs $E$ follows from the fact that if $(E',\beta(E',E))$ is separable, the space $E$ is weakly an $\aleph_0$-space, hence $E$ is separable,  see  \cite[Proposition 5.2]{GKKM}.
\end{proof}
It is known that  the Asplund property in the class of Banach spaces is inherited by closed subspaces, see \cite{Phelps} or  \cite{Fabian1}. Note that Corollary \ref{closed} shows that in general closed subspaces of Asplund lcs spaces need not be Asplund. The following positive assertion is another main result of our paper.
\begin{theorem}\label{seven}
Let $E$ be a quojection which is Asplund. If $F$ is a closed vector subspace of $E$ which is a quojection, then $F$ is Asplund.
\end{theorem}
\begin{proof}
Since $E$ is a quojection,  $E_p=(E/ker{p},\overline{p})$,  $p\in\mathcal{F}(E)$, is a Banach space isomorphic to  $E/ker(p)$, where $\mathcal{F}(E)$ is a countable family of continuous seminorms on $E$ generating the original  Fr\'echet  locally convex topology of $E$, and by  Theorem \ref{factor} we know that each Banach space $E_p$ is Asplund (recall that the quotient map from $E$ onto $E_p$ is bound covering). The family $\{p|F: p\in\mathcal{F}(E)\}$ forms a countable family of continuous seminorms generating the Fr\'echet locally convex topology of the subspace $F$ of $E$.
Note  that $F/ker(p|F)\subset E/ker(p)$ (algebraically and topologically) for $p\in\mathcal{F}(E)$. Since the subspace $F$ of $E$ is a quojection subspace,  each quotient space $F/ker(p|F)$ is a  Banach space (which is isomorphic to  $(F/ker(p|F), \overline{p|F})$). Hence
$F/ker(p|F)$ is a Banach subspace of the Asplund space $E/ker(p)$, so by Theorem \ref{separable-dual} one gets that $F/ker(p|F)$ is Asplund for each $p\in\mathcal{F}(E)$. Applying again Theorem \ref{factor} one gets that $F$ is Asplund.
\end{proof}
Last Theorem \ref{seven} may suggest the following
\begin{problem}\label{final}
Does there exist a quojection Fr\'echet lcs $E$  which is Asplund not being  a countable product of Banach spaces and yet $E$ contains no infinite-dimensional Banach subspaces?
\end{problem}
\begin{remark}
There exists a quojection Fr\'echet lcs $E$ (constructed by Ostrovskii  \cite[Theorem]{Ostrovskii}) which is not a countable product of Banach spaces and yet contains no infinite-dimensional Banach subspace. Ostrovskii observed that to prove his Theorem it was enough to construct a sequence $(E_{n})_n$ of Banach spaces such that for each $n\in\omega$ there exists a surjective continuous linear map $P_{n}:E_{n+1}\rightarrow E_{n}$ and for each $j,k\in\omega$, $j\neq k$, the spaces $E_{j}$ and $E_{k}$ do not have isomorphic infinite-dimensional subspaces. Therefore, this observation  combined with Theorem \ref{fifth} implies that to solve (positively) Problem \ref{final}, it is sufficient to have that this sequence $(E_n)_n$ will consist of Banach spaces which are Asplund spaces.
\end{remark}
\begin{problem}
Does there exist an infinite-dimensional  Fr\'echet lcs $E$ which is Asplund  that contains a Banach subspace $F$ which is not Asplund?
\end{problem}
%Every Fr\'echet lcs $E$ is the  projective limit of a sequence of locally Banach spaces $(E_n)_n$, so $F\subset E\subset \prod_{n}E_{n}$, \cite[Remark 24.5]{Meise-Vogt}.  By \cite[proof of Proposition 3.13]{Kakol-Leiderman}
%there exist  $E_{n_1},\dots, E_{n_{k}}$ from the  sequence $(E_{n})_n$  such that $F$ is isomorphic to  a subspace of $\prod_{j=1}^{k}E_{n_{j}}$. If each $E_n$ is Asplund, then $\prod_{j=1}^{k}E_{n_{j}}$  is Asplund, so $F$ %must be  Asplund. To deduce that each $E_n$ is Asplund (provided $E$ is Asplund), one may try to find conditions under which each $E_n$ is Asplund; for example one should  find  cases when  $E$ can be mapped on each $E_n$ by %a continuous linear map. This clearly holds if $E$ is  a quojection.

\vspace{0.3in}

\centerline{{\bf Statements and Declarations}}

The authors declare that no  support was received during the preparation of
this manuscript. The authors have no relevant financial or non-financial interests
to disclose.

\centerline{{\bf Conflict of interest}} 

There is no conflict of interest.

\centerline{{\bf Data availability}}

Data sharing not applicable to this article as no datasets were generated or
analyzed during the current study.
\vspace{0.3in}
\end{document}